\documentclass{proc-l}

\usepackage{amsmath}
\usepackage{hyperref}

\newtheorem{theorem}{Theorem}[section]
\newtheorem*{theorem*}{Theorem}
\newtheorem{lemma}[theorem]{Lemma}
\newtheorem{corollary}[theorem]{Corollary}

\theoremstyle{definition}
\newtheorem{definition}[theorem]{Definition}

\theoremstyle{remark}

\numberwithin{equation}{section}

\begin{document}

\title[Regularity of radicals and arithmetic degrees]{Upper bounds for regularity of radicals of ideals and arithmetic degrees}


\author{Yihui Liang}
\address{Department of Mathematics, Purdue University, 150 N. University Street, West Lafayette, IN47907, USA}
\curraddr{}
\email{liang226@purdue.edu}
\thanks{The author was partially supported by the Ross-Lynn Research Scholar Fund from Purdue University.}


\date{}

\commby{}


\begin{abstract}
Let $S$ be a polynomial ring in $n$ variables over a field. Let $I$ be a homogeneous ideal in $S$ generated by forms of degree at most $d$ with $\text{dim}(S/I)=r$. In the first part of this paper, we show how to derive from a result of Hoa an upper bound for the regularity of $\sqrt{I}$. More specifically we show that $\text{reg}(\sqrt{I})\leq d^{(n-1)2^{r-1}}$. In the second part, we show that the $r$-th arithmetic degree of $I$ is bounded above by $2\cdot d^{2^{n-r-1}}$. This is done by proving upper bounds for arithmetic degrees of strongly stable ideals and ideals of Borel type.
\end{abstract}

\maketitle

	\section{Introduction}
Let $S=\mathbb{K}[x_1,\dots,x_n]$ be a standard graded polynomial ring over a field. \linebreak Castelnuovo-Mumford regularity is an important homological invariant that measures the complexity of performing computations with finitely generated graded $S$-modules, for instance calculating graded free resolutions and the associated invariants. In \cite{Ravirad}, Ravi was the first to investigate the relationship between regularity of ideals and regularity of their radicals. Recall that the radical of an ideal $I$, denoted by $\sqrt{I}$, is defined as the set of all elements $x$ such that some power of $x$ lies in $I$. More specifically, Ravi showed that $\text{reg}(\sqrt{I})\leq \text{reg}(I)$ when $I$ is a monomial ideal, or when $R/\sqrt{I}$ is a Buchsbaum module, or in some cases when $\sqrt{I}$ defines a curve in $\mathbb{P}^3$. In addition, he raised the following question: Is it always true that $\text{reg}(\sqrt{I})\leq \text{reg}(I)$ for any homogeneous ideal $I$? This question is answered negatively by Chardin and D'Cruz \cite{chardincruz}, as they provided a family of ideals $I_{m,n}$ such that $\text{reg}(I_{m,n})=m+n+2$ but $\text{reg}(\sqrt{I})=mn+2$. It is natural to ask the next question: Does there exist a bound of $\text{reg}(\sqrt{I})$ in terms of $\text{reg}(I)$? To our best knowledge, so far there has not been any answer to the above question. 

In the first part of the paper, we explain how to specialize an existing formula of Hoa \cite{hoareg} to obtain the following bound (see Corollary \ref{regraddegreebd}). Recall that the generating degree is bounded above by the regularity so one may replace $d$ by $\text{reg}(I)$ in the inequalities below. 

\begin{theorem*}
	Let $I$ be a homogeneous ideal in $S$ of dimension $r\geq 2$ and generated by forms of degree at most $d$, then
	\[
	\text{reg}(\sqrt{I})\leq \left(\frac{d^{n-r}(d^{n-r}-1)}{2}+d^{n-1}\right)^{2^{r-2}} \leq d^{(n-1)2^{r-1}}.
	\]
\end{theorem*}

In the second part of the paper, we focus on proving upper bounds for arithmetic degrees of homogeneous ideals in terms of their generating degree (see $\S$4). Arithmetic degrees were introduced by Bayer and Mumford in \cite{computealggeo}, and they arise as refinements of multiplicities which serve as an important complexity measure. For this reason, it is desirable to find good bounds for this invariant. In the literature, there are two classical results of \cite{sturmfels} and \cite{hoa1998castelnuovo} which give bounds on arithmetic degrees for arbitrary monomial ideals in terms of their generating degrees. Since the arithmetic degrees of ideals are bounded above by the arithmetic degrees of their initial ideals, in particular the generic initial ideals, it is helpful to consider strongly stable ideals and ideals of Borel type.  For strongly stable ideals, bounds which depends on the primary decompositions of the ideals were given in \cite{doi:10.1080/00927872.2020.1726940}. Some bounds were also proved for square-free strongly stable ideals (see \cite{squarefree1} and \cite{squarefree2}). 

In Section 4.1, we first prove upper bounds for arithmetic degrees of strongly stable ideals and ideals of Borel type (see Corollary \ref{stronglytstablebd} and \ref{weaklystablebd}). Then in Section 4.2, we derive from our result on ideals of Borel type the following general bounds for homogeneous ideals (see Theorem \ref{arithdeghomideal}):

\begin{theorem*}
	Let $I$ be a homogeneous ideal in $S=\mathbb{K}[x_1,\dots,x_n]$ generated by forms of degree at most $d$. For all $r=0,\dots,n-1$, the $r$-th arithmetic degree of $I$ is bounded above by
	\[
	\text{arith-deg}_r(I)\leq 2\cdot d^{2^{n-r-1}}.
	\]
\end{theorem*}
Currently all the existing bounds for arithmetic degrees of arbitrary homogeneous ideals depend on their regularities, for example in \cite{computealggeo} Bayer and Mumford showed that $\text{arith-deg}_r(I)\leq \text{reg}(I)^{n-r}$. With their formula, to get a bound in terms of the generating degree, one needs to combine it with a doubly exponential bound of regularity (for instance \cite[Corollary 2.13]{caviglia2021linearly}) and the inequality becomes $\text{arith-deg}_r(I)\leq d^{(n-r)2^{n-r-1}}$. Notice that our bound is a significant improvement as we eliminate the factor of $n-r$ in the exponent.

\section{Preliminaries}

In this section, we provide definitions and some basic facts about multiplicity, geometric degree, arithmetic degree, Castelnuovo–Mumford regularity, strongly stable ideal, and ideal of Borel type. For a more thorough introduction to these subjects, see \cite{eisenbook}, \cite{computealggeo}, and \cite{herzog}. 

\subsection{Castelnuovo-Mumford regularity, multiplicity, geometric degree, and arithmetic degree}

Let $\mathbb{K}$ be an arbitrary field and $S$ be the polynomial ring $\mathbb{K}[x_1,\dots,x_n]$. All the invariants that will be introduced in this section do not change if we extend our field $\mathbb{K}$ to $\mathbb{K}(y)$ where $y$ is a new indeterminate, thus we may assume our field is infinite when needed.

Let $M$ be a finitely generated graded module over $S$. Let $\beta_{ij}(M):=$\linebreak$\text{dim}_{\mathbb{K}}(\text{Tor}^S_i(M,\mathbb{K})_j)$ be the graded Betti numbers of $M$. The \textit{Castelnuovo–Mumford regularity} of $M$ is defined as $\text{reg}(M)=\text{max}\{j : \beta_{i,i+j}(M)\neq 0 \text{ for some }i \}$.

For a graded module $M$, the \textit{multiplicity} of $M$, denoted by $e(M)$, is the normalized leading coefficient of the Hilbert polynomial of $M$. Let $I$ be a homogeneous ideal and $\prec$ be any monomial order, the initial ideal of $I$ with respect to $\prec$ is denoted by $\text{in}_\prec(I)$. Since the Hilbert function of $S/I$ agrees with the Hilbert function of $S/\text{in}_\prec(I)$, the multiplicities of $S/I$ and $S/\text{in}_\prec(I)$ agree as well.

For a homogeneous ideal $I$ in $S$, let $\text{Ass}(S/I)$ denote the set of associated primes of $S/I$ and $\text{Min}(S/I)$ denote the set of minimal primes of $S/I$. For a prime ideal $P$, the length-multiplicity of $P$ with respect to $I$, denoted by $\text{mult}_I(P)$, is defined as the length of the largest submodule of $(S/I)_P$ with finite length. 

Let us recall the associativity formula for multiplicity:
\[
e(S/I)=\sum_{P\in \text{Min}(S/I),\text{ dim}(S/P)=\text{dim}(S/I)}\text{mult}_I(P)\cdot e(S/P).
\]
Notice that for a minimal prime $P$, $(S/I)_P$ has finite length.

The \textit{$r$-th geometric degree} of $I$ is defined as:
\begin{equation*}
	\text{geom-deg}_r(I)=\sum_{P\in \text{Min}(S/I), \text{ dim}(S/P)=r}\text{mult}_I(P)\cdot e(S/P).
\end{equation*}
The \textit{geometric degree} of $I$, denoted by $\text{geom-deg}(I)$, is the sum of $\text{geom-deg}_r(I)$ for all $r$.

The \textit{$r$-th arithmetic degree} of $I$ is defined as:
\begin{equation*}
	\text{arith-deg}_r(I)=\sum_{P\in \text{Ass}(S/I), \text{ dim}(S/P)=r}\text{mult}_I(P)\cdot e(S/P).
\end{equation*}
The \textit{arithmetic degree} of $I$, denoted by $\text{arith-deg}(I)$, is the sum of $\text{arith-deg}_r(I)$ for all $r$. It is clear from the definitions that $e(S/I)\leq \text{geom-deg}(I)\leq \text{arith-deg}(I)$.

We present in Corollary \ref{geodegcoro} an upper bound for the geometric degree of $I$, which will be used in Section 3. This bound can be deduced from the theorem below (see \cite[Theorem 4.3]{sturmfels}). A bound for the geometric degree can also be obtained using \cite[Proposition 3.5]{computealggeo} which says the $i$-th geometric degree is bounded above by $d^{n-i}$ for all $i$. From this result one gets $\text{geom-deg}(I)= \sum_{i=1}^{r} \text{geom-deg}_i(I)\leq \sum_{i=1}^{r} d^{n-i}$, where $r$ is the dimension of $S/I$. Notice that Corollary \ref{geodegcoro} is a slight improvement of the above bound. 
\begin{theorem} [Sturmfels, Trung, Vogel] \label{geodeg}
	Let $J\subset I$ be homogeneous ideals in $S$ such that $(S/J)_P$ is Cohen-Macaulay for every minimal prime $P$ of $I$. Let $f_1,\dots,f_m$ be forms in $I$ with degrees $d_1\geq d_2\geq \cdots \geq d_m$ such that $I=J+(f_1,\dots,f_m)$. Then
	\[
	\text{geom-deg}(I)\leq d_1d_2\cdots d_t\cdot \text{geom-deg}(J),
	\]
	where $t:=\text{max}\{\text{ht}(P/J) : P\in\text{Min}(I) \}$.
\end{theorem}

\begin{corollary} \label{geodegcoro}
	Let $I$ be a homogeneous ideal in $S$ of dimension $r\neq 0$ generated by forms of degree at most $d$, then the geometric degree of $I$ is bounded above by:
	\[
	\text{geom-deg}(I)\leq d^{n-1}.
	\]
\end{corollary}

\begin{proof}
	Without loss of generality, we may assume the field is infinite.
	Let $f_1,\dots,f_s$ be homogeneous polynomials that minimally generate $I$ with $\text{deg}(f_i)\leq d$. We may assume $f_1,\dots,f_h$ form a regular sequence where $h=n-r$ is the height of $I$, and $f_{h+1},\dots,f_s$ are relabeled so that $\text{deg}(f_{h+1})\geq \cdots \geq \text{deg}(f_{s})$. Let $J=(f_1,\dots,f_h)$, so we have $I=J+(f_{h+1},\dots,f_s)$. By assumption the homogeneous maximal ideal cannot be a minimal prime of $I$, therefore we get $t:=\text{max}\{\text{ht}(P/J):P \in\text{Min}(I)\}\leq n-h-1$. Notice that the geometric degree of $J$ is equal to its multiplicity which is equal to $\prod_{i=1}^{h}\text{deg}(f_i)$. Hence we get $\text{geom-deg}(I)\leq \prod_{i=h+1}^{t+h}\text{deg}(f_i) \cdot \text{geom-deg}(J)\leq d^{n-h-1} \cdot \prod_{i=1}^{h}\text{deg}(f_i) \leq d^{n-h-1}\cdot d^{h}=d^{n-1}$.
	
\end{proof}

\subsection{Strongly-stable ideals and ideals of Borel type}	
Recall that a monomial ideal $I$ is an ideal of \textit{Borel type} (also called weakly stable or an ideal of nested type) if for every monomial $u\in I$, if $l$ is the maximum integer such that $x_i^l|u$, then for every $j\leq i$ there exists some $t\geq 0$ such that $x_j^tu/x_i^l\in I$. Another equivalent definition says $I$ is an ideal of Borel type if every associated prime of $I$ is of the form $(x_1,\dots,x_i)$ for some $i$.

A monomial ideal $I$ is \textit{strongly stable} if for every monomial $u\in I$, for any $i$ such that $x_i|u$ and $j\leq i$, we have $x_ju/x_i\in I$. 

Given a set of monomials $U=\{u_1,\dots,u_s\}$, one can consider the smallest strongly stable ideal $I$ that contains $U$. These $u_i$'s are sometimes called Borel generators of $I$ (see \cite{FRANCISCO2011522} for more information about Borel generators).

Let $I$ be a homogeneous ideal and $\prec$ be any monomial order such that $x_n\prec x_{n-1}\prec \cdots \prec x_1$. If $\mathbb{K}$ is infinite, then there exists a nonempty Zariski open set $U\subset \text{GL}_n(\mathbb{K})$ such that $\text{in}_\prec(\alpha I)=\text{in}_\prec(\alpha^\prime I)$ for all $\alpha,\alpha^\prime \in U$. The \textit{generic initial ideal} of $I$ with respect to $\prec$ is defined as $\text{gin}_\prec(I):=\text{in}_\prec(\alpha I)$ for any $\alpha \in U$. It is well-known that $\text{gin}_\prec(I)$ is an ideal of Borel type. Moreover if $\text{char}(\mathbb{K})=0$, $\text{gin}_\prec(I)$ is a strongly stable ideal.

\section{Regularity of radical of ideal}

Let $I$ be a homogeneous ideal in $S$ generated by forms of degree at most $d$. In this section we use a theorem of Hoa in \cite{hoareg} to derive an upper bound for the regularity of radical of $I$. The theorem of Hoa is given below:

\begin{theorem} [Hoa] \label{hoathm}
	Assume $S/J$ is a reduced ring of dimension $r\geq 2$ and multiplicity $e(S/J)$, then
	\[
	\text{reg}(J)\leq\left(\frac{e(S/J)(e(S/J)-1)}{2}+\text{arith-deg}(J)\right)^{2^{r-2}}.
	\]
\end{theorem}

We obtain the following upper bound of $\text{reg}(\sqrt{I})$ in terms of multiplicity and geometric degree of $I$ by applying the above inequality to $S/\sqrt{I}$.

\begin{corollary} \label{regradgeombd}
	Let $I$ be a homogeneous ideal in $S$ of dimension $r\geq 2$, then
	\[
	\text{reg}(\sqrt{I})\leq \left(\frac{e(S/I)(e(S/I)-1)}{2}+\text{geom-deg}(I)\right)^{2^{r-2}}.
	\]
\end{corollary}

\begin{proof}
	Since $\sqrt{I}$ is the intersection of all minimal primes of $I$, we get that $\text{Ass}(S/\sqrt{I})$\linebreak$=\text{Min}(S/\sqrt{I})=\text{Min}(S/I)$. This implies that $\text{arith-deg}(\sqrt{I})=\text{geom-deg}(\sqrt{I})$. Also we have $\text{dim}(S/\sqrt{I})=\text{dim}(S/I)=r$. For any minimal prime $P$ of $I$, notice that $(S/\sqrt{I})_P=S_P/PS_P$, so the length-multiplicity  $\text{mult}_{\sqrt{I}}(P)=l((S/\sqrt{I})_P)=1$ which is clearly bounded above by $\text{mult}_{I}(P)$. It follows that $e(S/\sqrt{I})\leq e(S/I)$ and $\text{geom-deg}(\sqrt{I})\leq \text{geom-deg}(I)$. We get the desired bound by combining Theorem \ref{hoathm} with the above inequalities.
\end{proof}

By applying Corollary \ref{geodegcoro}, we obtain the following upper bound for $\text{reg}(\sqrt{I})$ in terms of the generating degree of $I$. Note that the $0$-dimensional case is trivial since $\sqrt{I}$ is the homogeneous maximal ideal.
\begin{corollary} \label{regraddegreebd}
	Let $I$ be a homogeneous ideal in $S$ of dimension $r\neq 0$ and generated by forms of degree at most $d$, then
	\[
	\text{reg}(\sqrt{I})\leq \left.
	\begin{cases}
		d^{n-1}, & \text{if } r=1\\
		\left(\frac{d^{n-r}(d^{n-r}-1)}{2}+d^{n-1}\right)^{2^{r-2}}, & \text{if } r\geq 2\\
	\end{cases} \right\} \leq d^{(n-1)2^{r-1}}.
	\]
	
\end{corollary}

\begin{proof}
	First notice that $e(S/I)=\text{geom-deg}_r(I)$ by the associativity formula. The $r$-th geometric degree can be bounded by $\text{geom-deg}_r(I) \leq d^{n-r}$ according to \cite[Proposition 3.5]{computealggeo}.
	
	If $r=1$, then $S/\sqrt{I}$ is a one-dimensional reduced ring, therefore it is Cohen-Macaulay. Then apply \cite[Theorem 1.2]{rossi2005castelnuovo} to get $\text{reg}(\sqrt{I})\leq e(S/\sqrt{I})$. By the proof in the previous theorem we have $e(S/\sqrt{I})\leq e(S/I)\leq d^{n-1}$. 
	
	Now assume $r\geq2$. Combining Corollary \ref{regradgeombd} with the inequality $e(S/I) \leq d^{n-r}$ and Corollary \ref{geodegcoro}, we have:
	\begin{equation*}
		\begin{split}
			\text{reg}(\sqrt{I})
			&\leq \left(\frac{e(S/I)(e(S/I)-1)}{2}+\text{geom-deg}(I)\right)^{2^{r-2}}\\
			&\leq \left(\frac{d^{n-r}(d^{n-r}-1)}{2}+d^{n-1}\right)^{2^{r-2}}\\
			&\leq d^{(n-1)2^{r-1}}.
		\end{split}
	\end{equation*}
\end{proof}

\section{Upper bounds for arithmetic degrees}
In this section, our goal is to obtain upper bounds for the arithmetic degrees of any homogeneous ideal $I$. By \cite[Theorem 2.3]{sturmfels},  $\text{arith-deg}_r(I)\leq \text{arith-deg}_r(\text{in}_\prec(I))$ \linebreak $ \text{ for all }r=0,1,\dots,n$ and for any monomial order $\prec$, in particular the arithmetic degrees of $I$ can be bounded above by the arithmetic degrees of its generic initial ideal. Therefore the problem reduces to the case where $I$ is an ideal of Borel-type, and when $\text{char}(\mathbb{K})=0$ we may further assume $I$ is strongly stable. 
\subsection{Arithmetic degrees of strongly-stable ideals and ideals of Borel type}
Let us first consider the case where $I$ is simply a monomial ideal. Notice that every associated prime of $I$ is generated by a subset of the variables so it has the form $P_Z=(x_i : x_i\in \{x_1,\dots,x_n\}\setminus Z)$, and it follows that $e(S/P_Z)=1$. Therefore to compute the arithmetic degrees of $I$, it suffices to compute the length-multiplicities mult$_I(P_Z)$, which in fact have a combinatorial description due to Sturmfels, Trung, and Vogel (see \cite[\S 3]{sturmfels}). To see this we need to define the notion of standard pairs first.

\begin{definition} [Sturmfels, Trung, Vogel] \label{stdp}
	Let $u$ be a monomial in S, let $Z$ be a subset of the variables $\{x_1,\dots,x_n\}$. A pair $(u,Z)$ is called a \textit{standard pair} with respect to the monomial ideal $I$ if the following conditions hold:
	\begin{enumerate}
		\item $Z\cap $ supp($u$)=$\emptyset$,
		\item $u\mathbb{K}[Z]\cap I=\{0\}$,
		\item if $(u^\prime,Z^\prime)\neq (u,Z)$ is another pair such that $Z^\prime\cap $ supp($u^\prime$)=$\emptyset$ and $u^\prime\mathbb{K}[Z^\prime]\cap I=\{0\}$, then $u\mathbb{K}[Z]\not\subset u^\prime\mathbb{K}[Z^\prime]$.
	\end{enumerate}
	Let std$_r(I)$ denote the number of standard pairs of the form $(\cdot,Z)$ such that $|Z|=r$. 
\end{definition}

The following lemma from \cite[Lemma 3.3]{sturmfels} allows us to transform the problem of computing the length-multiplicities into counting the number of standard pairs.
\begin{lemma} [Sturmfels, Trung, Vogel]
	For any subset $Z\subseteq \{x_1,\dots,x_n\}$, \linebreak $\text{mult}_I(P_Z)$ equals the number of standard pairs of the form $(\cdot,Z)$. In particular $\text{arith-deg}_r(I)=\text{std}_r(I)$ for all $r=0,\dots,n$.
\end{lemma}

Now let us assume $I$ is an ideal of Borel type, so any associated prime of $I$ must have the form $P=(x_1,\dots,x_i)$ for some $i$. By the above lemma, computing the $r$-th arithmetic degree of $I$ is equivalent to counting the number of standard pairs of $I$ that have the form $(x_1^{c_1}\cdots x_{n-r}^{c_{n-r}},\{x_{n-r+1},\dots,x_n\})$. The following lemma gives us upper bounds on the number of such standard pairs.

Let $G(I)$ be the set of minimal monomial generators of $I$. For a monomial $u$, let Md$_i(u)=\text{max}\{c:x_i^c \text{ divides }u\}$. Denote $\text{Md}_i(I)=\text{max}\{\text{Md}_i(u): u\in G(I)\}$.
\begin{lemma}
	 Let $I$ be an ideal of Borel type. For any $j=1,\dots,n$, \linebreak if $(x_1^{c_1}\cdots x_j^{c_j},\{x_{j+1},\dots,x_n\})$ is a standard pair with respect to $I$, then $0\leq c_i\leq \text{Md}_i(I)-1$ for all $i=1,\dots,j$. In particular $\text{std}_r(I)\leq \text{Md}_1(I)\text{Md}_2(I)\cdots \text{Md}_{n-r}(I)$ for all $r=0,\dots,n-1$.
\end{lemma}

\begin{proof}
	Assume for contradiction that $c_i\geq \text{Md}_i(I)$ for some $i=1,\dots,j$, we claim that either $x_1^{c_1}\cdots x_i^{c_i}\in I$ or $x_1^{c_1}\cdots x_{i-1}^{c_{i-1}}\mathbb{K}[x_{i},\dots,x_n]\cap I=\{0\}$. Notice that if $x_1^{c_1}\cdots x_i^{c_i}\in I$, then $x_1^{c_1}\cdots x_j^{c_j}\in I$, so the first case violates conditon (2) in the definition of standard pairs. The second case violates condition (3) since $x_1^{c_1}\cdots x_{j}^{c_{j}}\mathbb{K}[x_{j+1},\dots,x_n]\subset x_1^{c_1}\cdots x_{i-1}^{c_{i-1}}\mathbb{K}[x_{i},\dots,x_n]$. To prove the claim, assume there exists a monomial $x_1^{c_1}\cdots x_{i-1}^{c_{i-1}} x_i^{p_i}\cdots x_n^{p_n}\in$ $x_1^{c_1}\cdots x_{i-1}^{c_{i-1}}\mathbb{K}[x_{i},\dots,x_n]\cap I$. Since $I$ is an ideal of Borel type, there exists a monomial of the form $x_1^{c_1}\cdots x_{i-1}^{c_{i-1}}x_i^{q_i}\in I$. Let $x_1^{c_1^\prime}\cdots x_{i-1}^{c_{i-1}^\prime}x_{i}^{q_{i}^\prime}$ be the monomial in the minimal generating set that divides $x_1^{c_1}\cdots x_{i-1}^{c_{i-1}}x_{i}^{q_{i}}$. Then we have $c_1^\prime \leq c_1,\dots,c_{i-1}^\prime \leq c_{i-1},$ and $q_i^\prime \leq \text{Md}_i(I)\leq c_i$. It follows that $x_1^{c_1^\prime}\cdots x_{i-1}^{c_{i-1}^\prime}x_{i}^{q_{i}^\prime}$ divides $x_1^{c_1}\cdots x_i^{c_i}$, therefore $x_1^{c_1}\cdots x_i^{c_i}\in I$.
	
\end{proof}

\begin{corollary} \label{weaklystablebd}
	Let $I$ be an ideal of Borel type in $S$ generated by monomials of degree at most $D$. Let $I_{[i]}$ denote the image of $I$ in $S/(x_{i+1},\dots,x_n)\cong \mathbb{K}[x_1,\dots,x_i]$ and $D(I_{[i]})$ denote the generating degree of $I_{[i]}$. Then $\text{arith-deg}_r(I)\leq \prod_{i=1}^{n-r} D(I_{[i]})$ for all $r=0,\dots,n-1$.
\end{corollary}

\begin{proof}
	Since $I$ is an ideal of Borel type, by \cite[Lemma 1.5]{cavigliasbarra} we have $\text{Md}_{i}(I_{[i]})=\text{Md}_i(I)$ for all $i=1,\dots,n$. Also it is clear that $\text{Md}_{i}(I_{[i]})\leq D(I_{[i]})$. The rest follows from the previous lemma.
\end{proof}

Now we further assume that $I$ is a strongly stable. Given a set of monomials $U=\{u_1,\dots,u_s\}$, we can obtain the smallest strongly stable ideal containing $U$ by adjoining to $U$ all the monomials that can be obtained by swapping variables as in the strongly stable definition and letting this new set be the generating set. Let us first consider the simplest case where $I$ is the smallest strongly stable ideal containing a single monomial (sometimes called principal Borel ideal), in this case we can list all possible standard pairs of $I$.

\begin{lemma} \label{strongstd}
	Let $I$ be the smallest strongly stable ideal containing a monomial $u=x_1^{l_1}\cdots x_n^{l_n}$. For all $j=1,\dots,n$, we have $x_1^{c_1}\cdots x_j^{c_j}\mathbb{K}[x_{j+1},\dots,x_n]\cap I=\{0\}$ if and only if $c_1\leq l_1-1$, or $c_1+c_2\leq l_1+l_2-1$, or $\cdots$, or $c_1+\cdots+c_j\leq l_1+\cdots+l_j-1$.
\end{lemma}

\begin{proof}
	Since $I$ is strongly stable and contains $u=x_1^{l_1}\cdots x_n^{l_n}$, it also contains all monomials of the form $x_1^{a_1}\cdots x_j^{a_j}x_{j+1}^{l_{j+1}}\cdots x_n^{l_n}$ for all $a_1,\dots,a_j\geq 0$ such that $a_1\geq l_1,a_1+a_2\geq l_1+l_2,\dots,$ and $a_1+\cdots+a_j\geq l_1+\cdots+l_j$. Notice that $x_1^{c_1}\cdots x_j^{c_j}\mathbb{K}[x_{j+1},\dots,x_n]\cap I=\{0\}$ if and only if for all such $a_1,\dots,a_j$ as above, either $c_1<a_1$, or $c_2<a_2$, or $\dots$, or $c_j<a_j$. The latter statement is equivalent to $c_1\leq l_1-1$, or $c_1+c_2\leq l_1+l_2-1$, or $\cdots$, or $c_1+\cdots+c_j\leq l_1+\cdots+l_j-1$.
\end{proof}

\begin{corollary}
	 Let $I$ be the smallest strongly stable ideal containing a monomial $u=x_1^{l_1}\cdots x_n^{l_n}$. Then the set of all standard pairs of $I$ is equal to \linebreak $\bigcup_{j=1,\dots,n}\{(x_1^{c_1}\cdots x_j^{c_j},\{x_{j+1},\dots,x_n\}):c_1+\cdots+c_i\geq l_1+\cdots+l_i \text{ for all } i=1,\dots,j-1 \text{ and } c_1+\cdots+c_j\leq l_1+\cdots+l_j-1 \}$.
\end{corollary}

\begin{proof}
	Notice that if $(x_1^{c_1}\cdots x_j^{c_j},\{x_{j+1},\dots,x_n\})$ is a standard pair, then by condition (3) in Definition \ref{stdp}, we must have $x_1^{c_1}\cdots x_i^{c_i}\mathbb{K}[x_{i+1},\dots,x_n]\cap I\neq \{0\}$ for all $1\leq i<j$. The rest follows from the previous lemma.
\end{proof}

In general for an arbitrary strongly stable ideal $I$, there exists a unique minimal set of monomials in $I$ such that $I$ is the smallest strongly stable ideal containing it. Such a set can often be much smaller than a minimal generating set. With information of this set, we can find upper bounds of the arithmetic degrees of $I$.

\begin{corollary} \label{stronglytstablebd}
	Let $I$ be the smallest strongly stable ideal containing a set of monomials $\{u_i=x_1^{l_{i_1}}\cdots x_n^{l_{i_n}}: i=1,\dots,s\}$. Then for all $r=0,\dots,n-1,
	$
	\[
	\text{arith-deg}_r(I)= \text{std}_r(I)\leq \binom{L_{n-r}+n-r-1}{n-r}
	\]
	where $L_{n-r}=\text{max}\{l_{i_1}+\cdots+l_{i_{n-r}}: i=1,\dots,s\}$.
	
\end{corollary}

\begin{proof}
	\sloppy For all $i=1,\dots,s$, let $I_i$ be the smallest strongly stable ideal containing $u_i$. Then $x_1^{c_1}\cdots x_j^{c_j}\mathbb{K}[x_{j+1},\dots,x_n]\cap I=\{0\}$ if and only if $x_1^{c_1}\cdots x_j^{c_j}\mathbb{K}[x_{j+1},\dots,x_n]\cap I_i=\{0\}$ for all $i$. Let $(x_1^{c_1}\cdots x_{n-r}^{c_{n-r}}, \{x_{n-r+1},\dots,x_n\})$ be a standard pair of $I$, then by Definition \ref{stdp} and the above equivalence we have that $x_1^{c_1}\cdots x_{n-r}^{c_{n-r}}\mathbb{K}[x_{n-r+1},\dots,x_n]\cap I_i=\{0\}$ for all $i$ and $x_1^{c_1}\cdots x_{n-r-1}^{c_{n-r-1}}\mathbb{K}[x_{n-r},\dots,x_n]\cap I\neq \{0\}$. Assume for contradiction that $c_1+\cdots+c_{n-r}\geq l_{i_1}+\cdots+l_{i_{n-r}}$ for all $i$, then by Lemma \ref{strongstd}, this implies that $x_1^{c_1}\cdots x_{n-r-1}^{c_{n-r-1}}\mathbb{K}[x_{n-r},\dots,x_n]\cap I_i=\{0\}$ for all $i$ and hence $x_1^{c_1}\cdots x_{n-r-1}^{c_{n-r-1}}\mathbb{K}[x_{n-r},\dots,x_n]\cap I=\{0\}$, which is a contradiction. Thus  for some $i=1,\dots,s$, we get $c_1+\cdots+c_{n-r}\leq l_{i_1}+\cdots+l_{i_{n-r}}-1$. In particular every standard pair $(x_1^{c_1}\cdots x_{n-r}^{c_{n-r}}, \{x_{n-r+1},\dots,x_n\})$ satisfies $c_1+\cdots+c_{n-r}\leq L_{n-r}-1$. Counting the number of all possible $c_1,\dots,c_{n-r}\geq 0$ gives us the desired upper bound.
\end{proof}

\subsection{Upper bounds for arithmetic degrees of homogeneous ideals}
In this subsection, we derive bounds for arithmetic degrees of homogeneous ideals in terms of their generating degrees by passing to their generic initial ideals. Since the generic initial ideals are ideals of Borel type, we can apply our result from the previous subsection to obtain the following bound.
\begin{theorem} \label{arithdeghomideal}
	Let $I$ be a homogeneous ideal in $S=\mathbb{K}[x_1,\dots,x_n]$ generated by forms of degree at most $d$. For all $r=0,\dots,n-1$, the $r$-th arithmetic degree of $I$ is bounded above by
	\[
	\text{arith-deg}_r(I)\leq 2\cdot d^{2^{n-r-1}}.
	\]
\end{theorem}

\begin{proof}
	Let $J=\text{gin}_{\text{revlex}}(I)$ be the generic initial ideal of $I$ with respect to the degree reverse lexicographical order. Then by \cite[Theorem 2.3]{sturmfels}, we have $\text{arith-deg}_r(I)\leq \text{arith-deg}_r(J)$. For any $i=1,\dots,n-r$, let $I_{\langle i \rangle}$ denote the image of $I$ in \linebreak $S/(l_{n},\dots,l_{i+1})\cong \mathbb{K}[x_1,\dots,x_i]$ where $l_{n},\dots,l_{i+1}$ are general linear forms. Notice that $D(J_{[i]})\leq \text{reg}(J_{[i]})= \text{reg}(I_{\langle i \rangle})$ for all $i$ by \cite[Remark 2.12]{zerogin}. When $i\geq 3$, we have $ \text{reg}(I_{\langle i \rangle})\leq d^{2^{i-2}}$ by the regularity bound given in \cite[Corollary 2.13]{caviglia2021linearly}. When $i=1,2$, we have $\text{reg}(I_{\langle 1 \rangle})\leq d$ and $\text{reg}(I_{\langle 2 \rangle})\leq2d-1$ by \cite[\S 2]{cavigliasbarra}. Combining the above inequalities with Corollary \ref{weaklystablebd}, we get that
	\[
	\begin{split}
		\text{arith-deg}_r(I)
		\leq \text{arith-deg}_r(J)
		&\leq \prod_{i=1}^{n-r} D(J_{[i]})\\
		&\leq d\cdot (2d) \cdot \prod_{i=3}^{n-r} d^{2^{i-2}}\\
		&= 2\cdot d^{2^{n-r-1}}.\\
	\end{split}
	\]
\end{proof}

\section*{Acknowledgement}

The author would like to thank her advisor Giulio Caviglia for proposing these problems.

\bibliographystyle{amsplain}
\bibliography{regrad}{}

\end{document}